\documentclass[11pt,twoside,a4paper]{amsart}
\usepackage{a4wide}
\usepackage{amsmath,amsfonts,amssymb}

\usepackage{amsthm}

\newcommand{\Z}{\mathbb{Z}}
\newcommand{\Q}{\mathbb{Q}}

\newtheorem{definition}{Definition}
\newtheorem{theorem}{Theorem}
\newtheorem{lemma}{Lemma}

\newtheorem{proposition}{Proposition}

\newtheorem{problem}{Problem}

\theoremstyle{remark}
\newtheorem{remark}{Remark}

\begin{document}
 \title[Simultaneous palindromes]{On Simultaneous palindromes}
%\subjclass[2010]{}

\author[A. B\'erczes]{Attila B\'erczes}
\thanks{The research was supported in part by the University of Debrecen,
and by grants K100339 and NK104208 of the Hungarian National Foundation for Scientific Research.
This work was partially supported by the European Union and the European Social Fund through project Supercomputer, the national virtual lab (grant no.: 
TAMOP-4.2.2.C-11/1/KONV-2012-0010).}
\address{A. B\'erczes \newline
         \indent Institute of Mathematics, University of Debrecen \newline
         \indent H-4010 Debrecen, P.O. Box 12, Hungary}
\email{berczesa\char'100science.unideb.hu}

\author[V. Ziegler]{Volker Ziegler}
\thanks{The second author was supported by the Austrian Science Fund (FWF) under the project P~24801-N26.}
\address{V. Ziegler\newline
\indent Johann Radon Institute for Computational and Applied Mathematics\newline
\indent Austrian Academy of Sciences\newline
\indent Altenbergerstr. 69\newline
\indent A-4040 Linz, Austria}
\email{volker.ziegler\char'100ricam.oeaw.ac.at}

\begin{abstract}
A palindrome in base $g$ is an integer $N$ that remains the same when its digit expansion in base $g$ is reversed. Let $g$ and $h$ be given 
distinct integers $>1$. In this paper we discuss how many integers are palindromes in base $g$ and simultaneously palindromes in base $h$.
\end{abstract}

\maketitle

\section{Introduction}
 Let $a,g\in \Z$ with $a\geq 0$ and  $g\geq 2$. If $a$ has a symmetric digit expansion in base $g$, i.e. $a$ read from left to right is the same as read from
right to left,
then we call $a$ a palindrome in base $g$. In particular, we will use the following definition

\begin{definition}
Let $a$ be a positive integer with $g$-adic digit expansion
$$a=\sum_{i=0}^k a_ig^i, \qquad \text{with}\;\; a_i\in\{0,1,\ldots,g-1\}, \;\; \text{and}\;\;  a_k\not=0$$
then we write
$$\overline{(a)_g}=\sum_{i=0}^k a_ig^{k-i}$$
for the digit reversed companion to $a$. We call $a$ a palindrome in base $g$, if we have $a=\overline{(a)_g}$.
\end{definition}

There is a rich literature on integers that are as well palindromes for some fixed base $g$ as well have some other property like being a square
\cite{Korec:1991},
a $k$-th power \cite{Hernandez:2006,Cilleruelo:2009}, almost a $k$-th power \cite{Simmons:1972,Luca:2008a}, member of a recurrence sequence \cite{Luca:2003d} or
some
other sequences (in case of arithmetic sequence see \cite{Col:2009}), a prime \cite{Banks:2004} and many other properties. Also some authors considered
the case of palindromes that are palindromes in two or more bases simultaneously. In particular, Goins \cite{Goins:2009} proved that there are only finitely
many palindromes in base $10$ with $d\geq 2$ digits and $N$ is at the same time a palindrome with $d$ digits in a base $b\not=10$ (for a similar result see also
\cite{Basic:2012}).
On the other hand Luca and Togb\'e \cite{Luca:2008a} proved that there are only finitely many binary palindrome which are decimal Palindromes of the form
$10^n\pm 1$.

In this paper we consider the following problem

\begin{problem}\label{Prob:Simul_Pal}
For which pairs of bases $(g,h)\in \Z^2$ with $2\leq h<g$ are there only finitely many positive integers that are simultaneously palindromes in base $g$ and
$h$.
\end{problem}

Note that the answer to this problem is negative if $g=h^k$ for some $k\geq 2$, since all integers of the form $g^n\pm 1$ are palindromes in base $g$ as well as
in base $h$.
Therefore we consider Problem \ref{Prob:Simul_Pal} only for bases $g,h$ such that $g$ is not a perfect power of $h$.

Unfortunately we cannot give an answer to this problem yet, but using ideas form Luca and Togb\'e \cite{Luca:2008a}, who proved the finiteness
of binary palindromes of the form $10^n\pm 1$, we were able to prove the following theorem.

\begin{theorem}\label{th1}
Let $2\leq h<g$ be integers and assume that $h|g$ and that $h$ and $g$ are multiplicatively independent. If $N=a g^n+\overline{(a)_g}$
is a palindrome in base $h$, then
\begin{multline*}
n\leq \max\left\{\frac{\log ga}{\log h},\frac{\log g(\log (agh))^2}{(\log 2)^3},1.91\cdot 10^7 \log a (\log\log a)^3, \right.\\
\left. \phantom{\frac gh} 5.11\cdot 10^{12}\log g \log (agh) (\log(\log g \log (agh)))^2\right\}.
\end{multline*}
\end{theorem}

Note that the result of Luca and Togb\'e \cite{Luca:2008a} can be derived form our Theorem \ref{th1} togehter with an extensive computer search.
In particular, if we put $h=2,g=10$ and $a$ some integer smaller than $10^6$ then Theorem \ref{th1} implies that $N=a10^n+\overline{(a)_{10}}$ can be a binary
palindrome only if $n\leq 2.65\cdot 10^{15}$. Our second result is:

\begin{theorem}\label{th2}
Let $2\leq h<g$ be fixed integers and assume that $h|g$ and that $h$ and $g$ are multiplicatively independent. For all $\epsilon>0$ there are at least
$\Omega_{\epsilon,g,h}(x^{1/2-\epsilon})$ palindromes $N\leq x$ in base $g$ that are not palindromes in base $h$. Moreover the constants involved in the
$\Omega$-term are explicitly computable.
\end{theorem}

This theorem means that for $h \mid g$ most numbers which are palindromes in base $g$ are {\it not palindromes} in base $h$.

The above Theorems \ref{th1} and \ref{th2} can both be deduced from the following lemma

\begin{lemma}\label{lem:Main}
Let $2\leq h <g$ be integers and assume $h|g$ and $h$ and $g$ are multiplicatively independent. Moreover, let $N$ be a palindrome in bases $g$ of the form
$$ag^n+\sum_{i=0}^{n-m-1} a_i g^i, \qquad a_i \in \{0,1,\ldots,g-1\},$$
where $n \geq m+1$ and $a=(a_{n-m-1}\ldots a_0)_g$ is an $(n-m)$-digit number in base $g$. This means that $N$ is a palindrome in base $g$ starting with the
digits of $a$ (in base $g$) followed by $m$ zeros.
If
\begin{multline*}
m>C(a,g,h,n):=\max\left\{\frac{\log ga}{\log h},\frac{\log g(\log (agh))^2}{(\log 2)^3},\right.\\
\left.\phantom{\frac gh} 142 (\log n)^2, 2.022\cdot 10^{10} \log g \log (agh) \log n \right\}
\end{multline*}
then $N$ cannot be a palindrome in base $h$.

More precisely assume that $m>(\log ga)/(\log h)$ and write
$\alpha=a/\overline{\left(\overline{(a)_g}\right)_h}$. If $g,h$ and $\alpha$ are multiplicatively independent then $N$ can be a palindrome in base $h$
only if
\begin{equation}\label{eq:mainlem1}\left|\log \alpha-k\log h+n\log g\right|<\frac{11}{9h^{-m}}.\end{equation}
If $\alpha^r=g^sh^{-t}$ for some integers $r,s,t$ not all zero, then $N$ is a palindrome in base $h$ only if
\begin{equation}\label{eq:mainlem2}\left|(n+s)\log g-(k+t)\log h\right|<\frac{11r}{9h^{-m}}.\end{equation}
\end{lemma}

Note that for proving Theorems \ref{th1} and \ref{th2} only the first part of the lemma is essential. However the second part of the lemma is useful if one
wants
to find all simultaneous palindromes of a special form, e.g. finding all binary palindromes of the form $a 10^n+\overline{(a)_{10}}$ for some fixed $a$.

In the next section we will give a proof of the fundamental Lemma \ref{lem:Main} and in Section \ref{Sec:ProofTh} we deduce Theorems \ref{th1} and \ref{th2}
from that lemma.
In Section \ref{Sec:Numeric} we present some numeric results on simultaneous palindromes in bases $2$ and $10$. In the last section we present some numeric
results for other bases.

\section{Proof of the main lemma}\label{Sec:MainLemma}

The purpose of this section is to prove Lemma \ref{lem:Main}. Therefore assume that $N$ is a palindrome in base $g$ as well as in base $h$. Let
$n_a=\left\lfloor (\log a)/(\log g) \right\rfloor+1$ be the number of digits of $a$ in base $g$. Since $N$ is a palindrome in bases $g$ as well as in base
$h$ and since $h|g$ we have
$$N \equiv \overline{(a)_g} \mod h^{m+n_a}.$$
Therefore we know the last $m+n_a$ digits of $N$ in base $h$, provided that $\overline{(a)_g}<h^{m+n_a}$.

\begin{lemma}\label{lem:loga<<m}
$\overline{(a)_g}<h^{m+n_a}$ if $m>(\log ga)/(\log h)$.
\end{lemma}

\begin{proof}
Note that $a$ and $\overline{(a)_g}$ have the same number of $g$-adic digits, i.e. $\overline{(a)_g}<ag$. Using the formula for $n_a$ we see that
$\overline{(a)_g}<h^{m+n_a}$ if
$ga<h^mh^{\log a/\log g +1}$, hence
$$\frac{\log ga}{\log h}-\frac{\log a}{\log g}-1<\frac{\log ga}{\log h}<m$$
implies $\overline{(a)_g}<h^{m+n_a}$.
\end{proof}

Therefore we assume from now on $m>(\log ga)/(\log h)$. Because $N$ is a palindrome in base $h$ we also know the first $m+n_a$ digits of $N$ in base
$h$. In particular, we have
$$N=\overline{\left(\overline{(a)_g}\right)_h}h^k+\sum_{i=0}^{k-\tilde m-1} b_i h^i,$$
where $\tilde m=m+n_a-\tilde n_a$ with $\tilde n_a$ denoting the number of digits of $\overline{\left(\overline{(a)_g}\right)_h}$ in base $h$, and where
$$k=\left\lfloor \frac{\log N -\log \overline{\left(\overline{(a)_g}\right)_h}}{\log h} \right\rfloor. $$
Since
$$m+n_a=\tilde m+\left\lfloor \frac{\log \overline{\left(\overline{(a)_g}\right)_h}}{\log h}\right\rfloor +1$$
we have
$$\tilde m = m+\left\lfloor \frac{\log \overline{(a)_g}}{\log g}\right\rfloor -\left\lfloor \frac{\log \overline{\left(\overline{(a)_g}\right)_h}}{\log
h}\right\rfloor.$$
This yields the following inequality for $N$:
$$\overline{\left(\overline{(a)_g}\right)_h}h^k+h^{k-\tilde m}>N>\overline{\left(\overline{(a)_g}\right)_h}h^k.$$
Dividing this inequality through $\overline{\left(\overline{(a)_g}\right)_h}h^k$ yields
\begin{equation}\label{IEq:Dioph}
 \left|\frac{ag^n}{\overline{\left(\overline{(a)_g}\right)_h}h^k}-1\right|<
 \frac{h^{-\tilde m}}{\overline{\left(\overline{(a)_g}\right)_h}}
\end{equation}

On the other hand
$$\frac{h^{-\tilde m}}{\overline{\left(\overline{(a)_g}\right)_h}}=h^{-m-\left\lfloor \frac{\log \overline{(a)_g}}{\log g}\right\rfloor +\left\lfloor
\frac{\log \overline{\left(\overline{(a)_g}\right)_h}}{\log h}\right\rfloor -\frac{\log \overline{\left(\overline{(a)_g}\right)_h}}{\log h}} \leq
h^{-m}\leq\frac 14$$
provided $m\geq 2$. Writing $\alpha:= a/\overline{\left(\overline{(a)_g}\right)_h}$ and using that the inequality $\log |1-x|\leq 11 x/9$ holds,
provided $x\leq 1/4$, which can easily be proved by a Taylor expansion with Cauchy's remainder term
from  equation \eqref{IEq:Dioph} we obtain
\begin{equation}\label{IEq:Log}
\left|\log \alpha-k\log h+n\log g\right|\leq \frac{11}{9h^{-m}},
\end{equation}
which is inequality \eqref{eq:mainlem1} in Lemma \ref{lem:Main}. Inequality \eqref{eq:mainlem2} is deduced from \eqref{IEq:Log} by multiplying it by $r$ and
noting that
$r\log \alpha= s\log g-t\log h$.

We distinguish now between two cases. The first case is that $\alpha$ is multiplicatively independent of $g$ and $h$ and the second is that $\alpha$ is
multiplicatively dependent of $g$ and $h$.
The first case requires lower bounds for linear forms in three logarithms (we will use a result due to Matveev \cite{Matveev:2000}) and in the second case our
inequality will reduce to
an inequality in linear forms in two logarithms, where sharper bounds are known (we will use a result due to Laurent et. al. \cite{Laurent:1995}). Unfortunately
using this result
will involve the prime decompositions of $g,h$ and $\alpha$.

We start with the first case. Let us state Matveev's theorem \cite{Matveev:2000}:

\begin{theorem}\label{Matveev:Th}
Denote by $\alpha_1,\ldots,\alpha_n$ algebraic numbers, not
$0$ nor $1$, by $\log\alpha_1,\ldots$, $\log\alpha_n$
determinations of their logarithms, by $D$ the degree over $\Q$ of
the number field $K = \Q(\alpha_1,\ldots,\alpha_n)$, and by $b_1,\ldots,b_n$
rational integers. Furthermore let $\kappa=1$ if $K$ is real
and $\kappa=2$ otherwise. Choose
$$A_i\geq \max\{D h(\alpha_i),|\log\alpha_i|\} \quad
(1\leq i\leq n),$$ where $h(\alpha)$ denotes the absolute logarithmic
Weil height of $\alpha$ and
$$B=\max\{1,\max \{ |b_j| A_j /A_n: 1\leq j \leq n \}\}.$$
Assume $b_n\not=0$ and
$\log\alpha_1,\ldots,\log\alpha_n$ are linearly independent
over $\Z$; then
$$\log |b_1\log \alpha_1+\cdots+b_n \log \alpha_n|\geq -C(n)C_0 W_0 D^2 \Omega,$$
with
\begin{gather*}
\Omega=A_1\cdots A_n, \\
C(n)=C(n,\kappa)= \frac {16}{n! \kappa} e^n (2n +1+2 \kappa)(n+2)
(4(n+1))^{n+1} \left( \frac 12 en\right)^{\kappa}, \\
C_0= \log\left(e^{4.4n+7}n^{5.5}D^2 \log(eD)\right), \quad W_0=\log(1.5eBD \log(eD)).
\end{gather*}
\end{theorem}

We will apply Theorem \ref{Matveev:Th} directly to \eqref{IEq:Log}. Obviously $\kappa=1$ and $D=1$, and we may choose $A_1=\log (agh)$ since
$a,\overline{\left(\overline{(a)_g}\right)_h} <agh$,
$A_2=\log h$ and $A_3=\log g$. Next we have to estimate $B$:

\begin{lemma}\label{lem:BoundB}
 $B<2n$ if $m\geq 2>(\log h)/(\log g)+1$.
\end{lemma}

\begin{proof}
First note that $2n>(\log \alpha)/(\log g)=|b_1|A_1/A_3$, since
$$m\geq 1>\frac{\log g+\log h}{2\log g}>\frac{\log \alpha}{2 \log g}$$
and $n\geq m$.

Furthermore, we have the inequality
\begin{multline*}
2n=\frac{2\log N-2\log a}{\log g}\\
>\frac{\log N-\log \overline{\left(\overline{(a)_g}\right)_h}-\log h-\log g+\log N-\log a}{\log g}\\
=k\frac{\log h}{\log g}+n-\frac{\log h}{\log g}-1>
k\frac{\log h}{\log g}=|b_2|\frac{A_2}{A_3}.
\end{multline*}
\end{proof}

Therefore we obtain $W=1.152 \log n$ provided $n>m\geq 10^6$. Now Theorem \ref{Matveev:Th} together with inequality \eqref{IEq:Log} yields
$$2.022\cdot 10^{10} \log g (\log agh)\log n > m,$$
which proves the first case. Note that the bound $2.022\cdot 10^{10} \log g (\log agh)\log n$ contains the bound $m\leq 10^6$ in any case.

Now we consider the second case. Since by assumption $\alpha,g$ and $h$ are multiplicatively dependent, but $g$ and $h$ are multiplicatively
independent thus there exist integers $r,s,t$ with greatest common divisor $1$ such that
$\alpha^r=g^sh^t$ with $r\neq 0$.

\begin{lemma}\label{lem:dependence}
\begin{gather*}
|r|\leq \frac{\log g \log h (\log agh)}{(\log 2)^3}, \quad |s|\leq \frac{\log h (\log agh)^2}{(\log 2)^3}, \quad
|t|\leq \frac{\log g (\log agh)^2}{(\log 2)^3}
\end{gather*}
\end{lemma}

\begin{proof}
Now let $p_1,p_2$ be primes that divide $gh$, let $e_{\beta,i}=v_{p_i}(\beta)$ for $i=1,2$
and $\beta\in\{g,h,\alpha\}$. Here $v_p(x)$ denotes the $p$-adic valuation of $x$. Further, assume that the vectors $(e_{g,1},e_{h,1})$ and $(e_{g,2},e_{h,2})$
are
linearly independent over $\Z^2$.  Note that this pair of primes exists since $g$ and $h$ are multiplicatively independent. A technical but easy computation
shows that
\begin{align*}
s=&(e_{h,1}e_{\alpha,2}-e_{h,2}e_{\alpha,1})e_{\alpha_2};\\
t=&(e_{g,1}e_{\alpha,2}-e_{g,2}e_{\alpha,1})e_{\alpha_2};\\
r=&(e_{g,2}e_{h,1}-e_{g,1}e_{h,2})e_{\alpha_2}
\end{align*}
is a solution and since $(e_{g,1},e_{h,1})$ and $(e_{g,2},e_{h,2})$ are linearly independent $r\neq 0$. The statement of the lemma now follows from the simple
estimates
\begin{align*}
e_{h,i}\leq& \frac{\log h}{\log 2}; & e_{g,i}\leq& \frac{\log g}{\log 2};
& |e_{\alpha,i}|\leq& \frac{\log (agh)}{\log 2};
\end{align*}
for $i=1,2$.
\end{proof}

We multiply inequality \eqref{IEq:Log} by $r$ and obtain
\begin{equation}\label{IEq:Log2}
\left|(n+s)\log g-(k+t)\log h\right|<\frac{r11}{9h^{-m}}
\end{equation}

This time we apply the following result due to Laurent et. al. \cite{Laurent:1995}:

\begin{theorem}\label{Th:twologs}
Let $\alpha_1$ and $\alpha_2$ be two positive, real, multiplicatively independent
elements in a number field of degree $D$ over $\Q$. For $i=1,2$, let $\log\alpha_i$ be any determination of the logarithm of
$\alpha_i$, and let $A_i>1$ be a real number satisfying
$$ \log A_i \geq \max \{h(\alpha_i),|\log\alpha_i|/D,1/D\}.$$
Further, let $b_1$ and $b_2$ be two positive integers.
Define
$$b'=\frac {b_1}{D\log A_2}+\frac {b_2}{D\log A_1} \quad \text{and}
\quad \log b=\max\left\{\log b'+0.14,21/D, \frac 12\right\}.$$ Then
$$ |b_2 \log\alpha_2-b_1 \log\alpha_1|\geq
\exp \left(-24.34 D^4(\log b)^2 \log A_1 \log A_2\right). $$
\end{theorem}

We choose $\alpha_1=g$ and $\alpha_2=h$, thus we have $D=1$. Put $\log A_1=\log g$ and $\log A_2=(\log h)/(\log 2)\geq 1$ and start with estimating $b'$:
\begin{align*}
b'=& \frac{(n+s)\log 2}{\log h}+\frac{k+t}{\log g}\\
<&\frac{2n}{\log h}+\frac 1{\log g}+\frac 2{\log h}+\stackrel{\leq \frac {s\log 2}{\log h}+\frac t{\log g}}{\overbrace{\left(\frac 1{(\log 2)^2}+\frac 1{(\log
2)^3}\right)(\log agh)^2}}\\
<&\frac{2n}{\log h}+6.3(\log agh)^2<\frac {4n}{\log h}<6n
\end{align*}
provided that $n>m>3.15\log h(\log agh)^2$. Note that the first inequality is true because of
\begin{multline*}
k=\left\lfloor\frac{\log N-\log \overline{\left(\overline{(a)_g}\right)_h}}{\log h} \right\rfloor\\
<\frac{\log a+(n+1)\log g-\log a+\log g+\log h}{\log h}=n\frac{\log g}{\log h}+\frac{2\log g}{\log h}+1.
\end{multline*}
The inequality for $b'$ now implies that we may choose
$$\log b=2\log n>\max\{\log n+\log 6+0.14,21\}$$
provided that $n>m>37000$. Now Theorem \ref{Th:twologs} yields
\begin{equation*}
141(\log n)^2 \log g \log h<m\log h+\log (9/11)-\log\left(\frac{\log g (\log agh)^2}{(\log 2)^3}\right).
\end{equation*}
Let us assume $n>m>(\log g) (\log agh)^2/(\log 2)^3$ and since we also assume $n>m>37000$ this last inequality turns into
$$ 142(\log n)^2 \log g>m.$$
Therefore we have completely proved Lemma \ref{lem:Main}.

Note that all our assumptions on $m$ made during the proof together with the bounds for $m$ cumulate in the lower bound of Lemma \ref{lem:Main}.

\section{Proof of Theorems \ref{th1} and \ref{th2}}\label{Sec:ProofTh}

We start with the proof of Theorem \ref{th1}. In this case $N=a g^n+\overline{(a)_g}$ and therefore
$$m=n-\left\lfloor\frac{\log a}{\log g}\right\rfloor>n-\frac{\log a}{\log g}-1.$$
In view of Lemma \ref{lem:Main} this implies: If $n>C(a,g,h,n)+(\log a)/(\log g)+1$, then $N$ is not a palindrome in base $h$.
Let us consider the two inequalities $n>142 (\log n)^2+(\log a)/(\log g)$ and
$n> 2.022\cdot 10^{10} \log g \log (agh) \log n+(\log a)/(\log g)$.
We note that the largest solution to $n=A\log n+B$ is smaller than the largest solution to $n/\log n=A+B/\log A$ and the largest solution to
$n/\log n=C$ is smaller than $x(\log x)^2$ provided $x>e^2$. Now let $A=2.022\cdot 10^{10} \log g \log (agh)$ and $B=(\log a)/(\log g)+1$, then
\begin{multline*}
A+\frac{B}{\log A}=2.022\cdot 10^{10} \log g \log (agh)+\frac{\frac{\log a}{\log g}+1}{23.73+\log(\log g\log (agh)}\\ <2.023\cdot 10^{10} \log g \log (agh).
\end{multline*}
Therefore the inequality $n> 2.022\cdot 10^{10} \log g \log (agh) \log n+(\log a)/(\log g)+1$ is fulfilled whenever
$$n>5.11\cdot 10^{12}\log g \log (agh) (\log(\log g \log (agh)))^2.$$

Similarly the largest solution to $n=A(\log n)^2+B$ is smaller than the largest solution to $n/(\log n)^2=A+B/(\log A)^2$ and the largest
solution to
$n/(\log n)^2=C$ is smaller than $x(\log x)^3$ provided $x>62$. This time we put $A=142$ and $B=(\log a)/(\log g)+1$ and
$$A+\frac{B}{(\log A)^2}<142+\frac{\log a}{24.56 \log g}+0.041.$$
If $a\leq 2$, then the lower bound for $n$ will be $17295$ which is absorbed by the much larger bound found in the paragraph above. Therefore we may assume that
$a\geq 3$ and obtain
$$A+\frac{B}{(\log A)^2}=142+\frac{\log a}{24.56 \log g}<130 \log a$$
and therefore the inequality $n>142 (\log n)^2+(\log a)/(\log g)$ is fulfilled if
$$n>1.91\cdot 10^7 \log a (\log\log a)^3.$$
Therefore the proof of Theorem \ref{th1} is complete.

We turn now to the proof of Theorem \ref{th2}. We consider palindromes described in Lemma \ref{lem:Main} with $a<g$ and $m=(\log n)^3$. Then by Lemma
\ref{lem:Main} we know
that for some constant $C_{g,h}$ depending only on $g$ and $h$ we have $m>C(a,g,h,n)$ (see Lemma \ref{lem:Main}) for all $n>C_{g,h}$, i.e. these palindromes
cannot be palindromes
in base $h$ if $n>C_{g,h}$. On the other hand there are $\gg x^{1/2-\epsilon}$ palindromes of the described form provided $1/(\log n)^3>\epsilon$, which
proves Theorem \ref{th2}.

\section{Numerical Considerations}\label{Sec:Numeric}

The purpose of this section is to consider the case $g=10$ and $h=2$ more closely and think of it as a model case. The aim is to find decimal palindromes that
are also
binary palindromes, however we did not find many such palindromes.

\begin{proposition}\label{Prop:SmallSimulPal}
Let $N<10^{18}$ be a palindrome in base $10$ which is also a binary palindrome, then $N$ is one of the following
$62$ palindromes:
\begin{gather*}
1,3,5,7,9,33,99,313,585,717,7447,9009,15351,32223,39993,53235,\\
53835,73737,585585,1758571,1934391,1979791,3129213,5071705,5259525,\\
5841485,13500531,719848917,910373019,939474939,1290880921,\\
7451111547,10050905001,18462126481,32479297423,75015151057,\\
110948849011,136525525631,1234104014321,1413899983141,\\
1474922294741,1792704072971,1794096904971,1999925299991,\\
5652622262565,7227526257227,7284717174827, 9484874784849,\\
34141388314143,552212535212255,1793770770773971,3148955775598413,\\
933138363831339, 10457587478575401, 10819671917691801, 18279440804497281,\\
34104482028440143, 37078796869787073, 37629927072992673, 55952637073625955, \\
161206152251602161, 313558153351855313.
\end{gather*}
\end{proposition}

\begin{proof}
For all $a<10^{10}$ we construct decimal palindromes $N<10^{18}$ with an even number of digits by reversing the digits of $a$ and appending the reversed string
of digits at the string of digits of $a$,
i.e. if $a=\sum_{i=0}^{n} a_i 10^i$ we compute the palindrome
$$N=\sum_{i=0}^{n}a_i 10^{n+1+i}+\sum_{i=0}^n a_{n-i} 10^i.$$
Similarly we construct for all $a<10^{9}$ palindromes $N<10^{18}$ with an odd number of digits by
$$N=\sum_{i=0}^{n}a_i 10^{n+i}+\sum_{i=0}^{n-1} a_{n-i} 10^i.$$
With this procedure we have a complete list of all decimal palindromes $N<10^{18}$.

Now we test for each decimal palindrome $N$ whether it is a binary palindrome by the following algorithm. First we compute the number of binary digits
$$k=\left\lfloor \frac{\log N}{\log 2}\right\rfloor +1.$$
Let us put $n_k=n_0=N$ and compute subsequently for all $0\leq i \leq \lfloor k/2 \rfloor$ the $i$-th highest and $i$-th lowest binary digits of $N$
$$d_{k-i}=\left\lfloor \frac{n_{k-i}}{2^{k-i}}\right\rfloor, \qquad d_i=n_i \mod 2$$
and
$$n_{k-i-1}=\frac{n_{k-i}-2^{k-i}d_{k-i}}2,\qquad n_{i+1}=\frac {n_i-d_i}2$$.
If $d_{k-i}\neq d_i$ for some $i$ then $N$ is not a binary palindrome.
If $d_{k-i}=d_i$ for all $0\leq i \leq \lfloor k/2 \rfloor$ then $N$ is also a binary palindrome.

Note that if $N$ is not a binary palindrome we do not have to compute all binary digits of $N$ and in many cases after computing a few digits of $N$ will yield
a result.
Indeed implementing this algorithm in sage \cite{sage} and running it on a usual workstation we used about ?? hours of CPU time.
\end{proof}

\begin{remark}
We want to note that in the On-Line Encyclopedia of Integer Sequences \cite{OEIS} the list of the palindromes in bases $2$ and $10$ as sequence A060792. 
However, the list includes only the simultaneous palindromes up to $7451111547$.
\end{remark}

\begin{proposition}\label{Prop:smalla}
Let $N=a10^n+\overline{(a)_{10}}$ be a binary palindrome with $10\nmid a$ and $a<\min\{10^6,10^n\}$, then it is already
contained in the list of palindromes in Proposition~\ref{Prop:SmallSimulPal}.
\end{proposition}

In order to prove this proposition we have to consider the Diophantine inequalities \eqref{IEq:Log} and \eqref{IEq:Log2}. An upper bound for $m$ is given by
Lemma \ref{lem:Main} but this bound is very huge. Therefore we will use continued fractions in case of \eqref{eq:mainlem2} and a method due to Baker and
Davenport \cite{Baker:1969} to reduce the upper bound in case of \eqref{eq:mainlem1}. Let us state a variant of this reduction method:

\begin{lemma}\label{Lem:BakDav1}
Given a Diophantine inequality of the form
\begin{equation}\label{IEq:BakerDav}
|n_1+n_2\epsilon+\delta|<c_1\exp(-n_2 c_2)
\end{equation}
Assume $n_2<X$ and assume that there is a real number $\kappa>1$ and also assume there exists a convergent $p/q$ to $\epsilon$ with
$X/q<1/(2\kappa)$ such that
$$\|q \epsilon\|<\frac 1{2 \kappa X} \quad \text{and} \quad \|q \delta\|>\frac 1{\kappa},$$
$\|q\delta\|>\frac 1/\kappa$, where $\|\cdot\|$ denotes the distance to the nearest integer. Then we have
$$n_2\leq\frac{\log (2 \kappa q c_1)}{c_2}.$$
\end{lemma}

\begin{proof}
We consider inequality \eqref{IEq:BakerDav} and we multiply it by $q$. Then under our assumptions we obtain
$$c_1q\exp(-n_2 c_2)>|qn_1+n_2q\epsilon+q\delta|
\geq \left|\|n_2q\epsilon\|-\|q\delta\|\right|> \frac1{2\kappa}.$$
The last inequality holds since $p/q$ is a convergent to $\epsilon$ and therefore $|\epsilon q-p|< 1/q$.
Solving this inequality for $n_2$ we obtain the lemma.
\end{proof}

\begin{proof}[Proof of Proposition \ref{Prop:smalla}]
Now let us apply the second part of Lemma \ref{lem:Main} to the present situation. Since we assume $a< 10^6$ we have $m\geq n-(\log a)/(\log 10)> n-6$. And
therefore
either $n>30$ or one of the two inequalities \eqref{eq:mainlem1} and \eqref{eq:mainlem2} are fulfilled with $n\leq 2.65 \cdot 10^{15}:=X$ - since the upper
bound for $n$ obtained in
Theorem \ref{th1}. We have to distinguish between the two cases $\alpha^r=10^s2^t$ or $\alpha$, $2$ and $10$ are multiplicatively independent.

In the second case we divide by $\log 2$ and obtain inequality \eqref{IEq:BakerDav}, with $n_1=k$, $n_2=n$, $\epsilon=(\log 10)/(\log 2)$,
$\delta=(\log \alpha)/(\log 2)$, $c_1=(11\cdot2^6)/(9\log 2)$ and $c_2=\log 2$. With this choice we apply Lemma \ref{Lem:BakDav1}. Since $\epsilon$ is
independent of $a$
we can precompute suitable pairs $(q,\kappa)$ which may be applied to our situation. Therefore we compute the first $50$ convergents to $(\log 10)/(\log
2)$. For each convergent $p/q$
with $q>1.06 \cdot 10^{16}$ we form the pair $(q,\kappa)$with $\kappa:=2X/q$ and get a list of $16$ potential pairs applicable to Lemma \ref{Lem:BakDav1}. Now
let us fix $a$.
We subsequently test whether in our list is a pair $(q,\kappa)$ such that $\|q\delta\|>1/\kappa$, hence by Lemma \ref{Lem:BakDav1} we get a new bound
that should be rather small
and indeed in all cases our new bound yields $n\leq 81$. Further we want to emphasize that it is highly improbable that for a given $a$ no pair of our list of
candidates yields an application
of Lemma \ref{Lem:BakDav1} and therefore no new upper bound for $n$. Therefore we are left to test all remaining $n$ for our fixed $a$, which can be done by a
quick computer search.

In case of $\alpha^r=10^s2^t$ for some integers $r,s,t$ with $r^2+s^2+t^2\neq 0$ we know that we can choose $r=1$. Indeed the free $\Z$-Module generated by
$\{\log 10,\log 2\}$ is the same
as the free $\Z$-Module generated by $\{\log 5,\log 2\}$ and $\log \alpha$ is contained in the later one and therefore also in the first $\Z$-Module. Now we
obtain by Lemma \ref{lem:Main}
and in particular by inequality \eqref{eq:mainlem2}
$$\left|\frac{\log 10}{\log 2}-\frac{k+t}{n+s}\right|<\frac{11\cdot 2^6}{9\cdot 2^{-n}(n+s)\log 2}.$$
Note that therefore $(k+t)/(n+s)$ has to be a convergent to $(\log 10)/(\log 2)$ unless $n\leq 30$ or
$$\frac{1}{2(n+s)^2}<\frac{11\cdot 2^8}{9\cdot 2^{n}(n+s)\log 2}.$$
Let us note that this inequality does not hold for large $n$, in particular in all cases that we consider we can choose the bound $n\geq 30$. Therefore we know
that $n+s$ has to
be a multiple of $q$, where $p/q$ is a convergent to $(\log 10)/(\log 2)$, i.e. $n=kq-s$ for some positive integer $k$.
But, already for rather small $k$ and fixed convergent $p/q$ this choice will contradict the inequality
\begin{equation}\label{IEq:convergents}
 \left|\frac{\log 10}{\log 2}-\frac pq\right|<\frac{11\cdot 2^6}{9\cdot 2^{kq-s}(kq)\log 2}.
\end{equation}
We claim that inequality \eqref{IEq:convergents} is never fullfilled for $n\geq 30$. Since Lemma \ref{lem:dependence} we know $s\leq 34$ and therefore we may
assume $q<2.66\cdot 10^{15}$.
In particular we have to prove inequality \eqref{IEq:convergents} for $32$ convergents. If we replace in \eqref{IEq:convergents} the quantities $kq-s=n$ and
$kq=n+s$ by $30$ - we may do so
since we assume $n\geq 30$ - then inequality \eqref{IEq:convergents} is never satisfied by the first $8$ convergents. For the remaining $24$ cases we replace in
\eqref{IEq:convergents} $kq-s=n$
by $q-34$ and $kq$ by $q$. If this new inequality still holds also \eqref{IEq:convergents} holds. A quick computation using a computer algebra system like sage
\cite{sage} resolves this case.

Therefore we also have in this case a very efficient method to find all simultaneous palindromes, i.e. we only have to test all $n\leq 30$.

We implemented the idea above in sage \cite{sage} and computed for all $2< a < 10^6$ with $10\nmid a$ and $\overline{(a)_{10}}$ is odd all $n$ with
$a<\min\{10^6,10^n\}$ such that
$N=a10^n+\overline{(a)_{10}}$ is a binary palindrome. Note that in case of $\overline{(a)_{10}}$ is even then the last binary digit of $N$ would be $0$ and $N$
would not be a binary palindrome.
The computer search took on a single PC about $80$ minutes.
\end{proof}

\section{Other bases}\label{Sec:Problem}

In this section we want to discuss Problem \ref{Prob:Simul_Pal} for further base pairs $(h,g)$. In case of $(h,g)=(2,10)$ the preceeding sections show that
there are only few integers that are simultaneously
palindromes in bases $2$ and $10$. Looking at our results we even guess that there are only finitely many simultaneous palindromes for the bases
$2$ and $10$. In this last section we want to present shortly our
numeric considerations for other base pairs. In particular, we considered the pairs $(2,3)$, $(6,15)$, $(5,7)$, $(11,13)$ and $(7,29)$ and counted the number
$N$ of simultaneous palindromes
smaller than some bound $B$. Our results are listed in Table \ref{Tab:SimPal} below.
The algorithms were implemented in sage \cite{sage} and were run on a single PC.

\begin{table}[ht]
 \caption{Number of simultaneous palindromes}\label{Tab:SimPal}
 \begin{tabular}{|c|c|r|r|r|}
  \hline $g$ & $h$ & $B$ & $N$ & Time \\ \hline\hline
  $2$ & $3$ & $3^{66}\simeq  3.09\cdot10^{31}$ & $9$ & \\
  $6$ & $15$ & $6^{20}\simeq 3.66 \cdot 10^{15}$ & $58$ & 1d 5h\\
  $5$ & $7$ & $5^{24}\simeq 5.96\cdot 10^{16}$ & $57$ & 7d 5h\\
  $11$ & $13$ & $13^{14}\simeq 3.94 \cdot 10^{15}$ & $58$ &  23h\\
  $7$& $29$ & $7^{20}\simeq 7.98 \cdot 10^{16}$ & $73$ & 7d 3h \\
  $2$ & $10$ & $10^{18}$ & $62$ & 31d 17h \\\hline
 \end{tabular}
\end{table}

Let us note that the case $g=2$ and $h=3$ is included in the On-Line Encyclopedia of Integer sequences \cite{OEIS} as sequence A060792.

Looking at Table \ref{Tab:SimPal} the number of palindromes to the bases $2$ and $3$ simulaneously is very small and indeed we are led by our numeric
computations to the following problem:

\begin{problem}
Are there only finitely many positive integers that are palindromes in bases $2$ and $3$ simultaneously? If yes, how many are there? If there are 
infinitely many, find an asymptotic formula for the number of positive integers $\leq N$ that are palindromes in bases $2$ and $3$ simultaneously.
\end{problem}

\section*{Acknowledgement}
We want to thank Neal Sloan who pointed out that already a list of simultaneous palindromes in bases $2$ and $3$ respectively $2$ and $10$ exist in \cite{OEIS}.

The research was supported in part by the University of Debrecen,
and by grants K100339 and NK104208 of the Hungarian National Foundation for Scientific Research.
This work was partially supported by the European Union and the European Social Fund through project Supercomputer, the national virtual lab (grant no.:
TAMOP-4.2.2.C-11/1/KONV-2012-0010).
The second author was supported by the Austrian Science Fund (FWF) under the project P~24801-N26.

\newcommand{\etalchar}[1]{$^{#1}$}
\def\cprime{$'$}

\label{lastpage-01}

\end{document}